\documentclass[11pt]{amsart}
\usepackage{geometry,graphicx}

\usepackage{latexsym}
\usepackage{amssymb}
\usepackage[cp850]{inputenc}
\usepackage{epsfig}
\usepackage{epstopdf}
\usepackage{psfrag}
\usepackage{amsthm}
\usepackage{amscd}
\usepackage{amsmath}
\usepackage{amsfonts}
\usepackage{graphics}
\usepackage{enumerate}
\usepackage{color}
\theoremstyle{theorem}

\newtheorem{theorem}{Theorem}

\newtheorem{proposition}{Proposition}

\newtheorem{lemma}{Lemma}

\theoremstyle{definition}

\newtheorem*{remark}{Remark}

\newcommand{\BB}{{\mathcal B}}

\newcommand{\T}{{\mathcal T}}

\newcommand{\MM}{{\mathcal M}}

\newcommand{\com}[1]{{\rm Comb}_{#1}}

\newcommand{\commodgk}{{\rm Comb}_N{\rm Mod}_{g,k}}

\DeclareMathOperator{\Vol}{Vol}
\DeclareMathOperator{\link}{link}
\DeclareMathOperator{\Card}{Card}

\numberwithin{equation}{section}

\begin{document}

%




%

\title{Pants decompositions of random surfaces}




\author[L. Guth]{Larry Guth}

\address[Larry Guth] {Department of Mathematics, University of Toronto\\
 Toronto, Canada}

\email{lguth@math.utoronto.ca}

\author[H.~Parlier]{Hugo Parlier${}^\dagger$}
\address[Hugo Parlier]{Department of Mathematics, University of Fribourg\\
  Switzerland}
\email{hugo.parlier@gmail.com}
\thanks{${}^\dagger$Research supported by Swiss National Science Foundation grant number PP00P2\textunderscore 128557}

\author[R. Young]{Robert Young}
\address[Robert Young]{Courant Institute of Mathematical Sciences, New
York University}
\email{rjyoung1729@gmail.com}

\date{\today}



\begin{abstract} Our goal is to show, in two different contexts, that ``random" surfaces have large pants decompositions. First we show that there are hyperbolic surfaces of genus $g$ for which any pants decomposition requires curves of total length at least $g^{7/6 - \varepsilon}$. Moreover, we prove that this bound holds for most metrics in the moduli space of hyperbolic metrics equipped with the Weil-Petersson volume form. We then consider surfaces obtained by randomly gluing euclidean triangles (with unit side length) together and show that these surfaces have the same property.
\end{abstract}




%



\maketitle


Any surface of genus $g$, $g\ge 2$, can be decomposed into three-holed
spheres (colloquially, pairs of pants). We say that a surface has pants
length $\le l$ if it can be divided into pairs of pants by curves each of length
$\le l$. We say that a surface has total pants length $\le L$ if it can be
divided into pairs of pants by curves with the sum of the lengths 
$\le L$. The pants length and total pants length measure the size and complexity of
a surface. In particular, they describe how hard it is to divide the surface into
simpler parts. One of the main open problems in this area is to understand
how big the pants length of a genus $g$ hyperbolic surface can be. It
would also be interesting to understand how big the total pants length of
a genus $g$ hyperbolic surface can be.
In this paper, we use a random construction to find hyperbolic surfaces with
surprisingly large total pants length.

To put the paper in context, we review the known results about pants length
and total pants length. In \cite{BersDegen,BersIneq}, Bers proved that
for each genus $g$, the supremal pants length of a genus $g$ hyperbolic surface is finite.
This result is non-trivial because the moduli space of hyperbolic surfaces is
not compact, and Bers's result gives information about the geometry of surfaces
near the ends of moduli space. The supremal pants length of a genus $g$
hyperbolic surface is called the Bers constant, $\BB_g$. Later work
gave explicit estimates for the Bers constant. In \cite{BuserSep}, 
Buser and Sepp\"al\"a proved that every genus $g$ hyperbolic surface
has pants length at most $C g$ (where $C$ is a constant independent of $g$).
On the other hand, Buser \cite{BuserHab} gave examples of hyperbolic surfaces with pants
length at least $c g^{1/2}$ for arbitrarily large $g$. Buser conjectured that the Bers constant
of a hyperbolic surface is bounded by $C g^{1/2}$ \cite{BuserBook}. 

The total pants length has not been studied as much as the pants length, but
it also seems like a natural invariant. Since a pants decomposition has
$3g - 3$ curves in it, the estimate of Buser and Sepp\"al\"a implies that every
genus $g$ hyperbolic surface has total pants length at most $C g^2$. This is
the best known general upper bound.  In the other direction, it is easy to construct hyperbolic surfaces 
with total pants length at least $c g$ for every $g$ by taking covers of a genus
2 surface. The only previous non-trivial estimate comes from work of Buser and Sarnak on the geometry of certain arithmetic surfaces \cite{BuserSarnak}. They proved that there exist families of surfaces, one in each genus $g$, with the property that every topologically non-trivial curve has length at least $\sim \log g$. Since each curve in a pants decomposition is non-trivial, the total pants length of these arithmetic hyperbolic surfaces is at least $\sim g \log g$.

Now we have the background to state our main theorem.

\newtheorem{introth}{Theorem} 

\begin{introth} For any $\varepsilon > 0$, a ``random'' hyperbolic surface of genus $g$
has total pants length at least $g^{7/6 - \varepsilon}$ with probability tending to 1 as $g \rightarrow \infty$.
In particular, for all sufficiently large $g$, there are hyperbolic surfaces with total pants length at
least $g^{7/6 - \varepsilon}$.
\end{introth}

(To define a ``random'' hyperbolic surface we need a probability measure on the
moduli space of hyperbolic metrics. We use the renormalized Weil-Petersson volume
form. We discuss this notion of randomness more below.)

As another piece of context for our result, we mention the analogous questions for hyperbolic
surfaces with small genus but many cusps.
For simplicity, let's consider complete hyperbolic surfaces with genus 0 and $n$ cusps for
$n \ge 3$. These surfaces also have pants decompositions, and pants length is defined in the same way.
The arguments of Buser and Sepp\"al\"a show that such a surface has pants length at most $C n$, and Buser's
conjecture was that the pants length should be at most $C n^{1/2}$. Balacheff and Parlier proved this
conjecture in \cite{BalPar}.
In a more general context in \cite{bal-par-sab},
it is shown how to recuperate these results via Balacheff and Sabourau's diastolic inequality \cite{BalacheffSabourau}.
There are also very good estimates for the total pants length in this context.
Balacheff, Parlier and Sabourau showed that the total pants length of a hyperbolic surface with
genus 0 and $n$ cusps is at most $C n\log n$.  It's easy to find examples where the total pants length
is at least $c n$, so their bound is sharp up to logarithmic factors.  
The same authors went on to show that hyperelliptic genus $g$ hyperbolic surfaces have total pants length at most $\sim g \log g$. 
The hyperbolic surfaces in Theorem 1 are very different from hyperelliptic surfaces or from 
surfaces of genus 0 with many cusps.

Our lower bound is a lot stronger than the one coming from the Buser-Sarnak estimate. Instead of 
improving the trivial bound by a factor of $\log g$, we improve it by a polynomial factor $g^{1/6 - \varepsilon}$.
Let's take a moment to explain why it is difficult to prove such a lower bound.
Almost all the random hyperbolic surfaces we construct have diameter around $\log g$. Therefore, they have
lots of non-trivial curves with length around $\log g$. For example, we can make a basis for the first homology
of the surface using curves of length around $\log g$. So there are lots of short curves that look like good
candidates to include in a pants decomposition. The key issue seems to be that the curves in a pants
decomposition need to be disjoint. After we pick a first curve, the second curve needs to avoid it. This cuts down
our options for the second curve. Then
the third curve needs to avoid the first two curves, so we have even fewer options.
 If the pants length of the surface
is near $g^{7/6}$, then the curves ``get in each other's
way'' so much that we have to pick dramatically longer curves to finish building a pants
decomposition. This kind of effect looks difficult to prove because there are so many choices about how to choose
the early curves. At the present time, for any particular hyperbolic surface, we cannot prove that the total
pants length is any larger than $g \log g$. But for a random hyperbolic surface, the pants length is usually
much larger than $g \log g$. 

Our proof is essentially a counting argument. If a surface has total
pants length at most $L$, we can construct it by gluing some hyperbolic pairs of pants
with total boundary length $\le L$. A hyperbolic pair of pants is
determined by its boundary lengths, so the number (really,
Weil-Petersson volume) of possible surfaces with total pants length
$\le L$ is governed by the number of possible ways of choosing these
hyperbolic pairs of pants and gluing them together. We estimate this
volume and show that if $L\le g^{7/6-\varepsilon}$, then it is much
smaller than the total volume of moduli space.

To define the random hyperbolic surfaces above, we use the Weil-Petersson volume form on moduli space.
Moduli space and the Weil-Petersson volume form are both fairly abstract, making it particularly hard to 
visualize what is going on. To help make things more down-to-earth, we will also consider an
elementary random construction: randomly gluing together equilateral triangles to form a closed surface.
Suppose that we take $N$ Euclidean triangles of side length 1, where
$N$ is even. These triangles can be glued together to make a surface,
and Gamburd and Makover \cite{GamMak} showed that this surface usually has genus close to $(N+2)/4$. We consider the set of triangulated surfaces obtained in this way. We identify surfaces that
are simplicially isomorphic, and the equivalence classes form a kind of combinatorial ``moduli space''. We put the uniform measure on this finite set, so we can speak of
a random combinatorial surface. (This definition goes back to Brooks and Makover \cite{Brooks-Makover}, who
studied the first eigenvalue of the Laplacian and the systole of random combinatorial surfaces.) 

\begin{introth} For any $\varepsilon > 0$, a random combinatorial surface
 with $N$ triangles has total pants length
at least $N^{7/6 - \varepsilon}$ with probability tending to 1 as $N \rightarrow \infty$.
\end{introth}
The proof of Theorem 2 is closely parallel to the proof of Theorem 1.
In some ways, the proof is not as clean, but on the other hand the
proof is completely elementary and self-contained. 

In a similar
spirit, Brooks and Makover suggested that for large $N$, the counting measure on the combinatorial moduli
space above may be close
to the Weil-Petersson measure on the moduli space of hyperbolic metrics. We're not sure how to
phrase this in a precise way, but there does seem to be a strong
analogy between combinatorial surfaces with $N$ triangles and
hyperbolic surfaces of genus $g$ when $N=g/4$. For example, there
are roughly $N^{N/2}\approx g^{2g}$ different surfaces with $N$ triangles,
comparable to the volume of moduli space, which is roughly $g^{2g}$.

The combinatorial metrics we study in Theorem 2 are not hyperbolic, so let us mention what is known
about the pants lengths of arbitrary metrics on a surface of genus $g$. Here the natural question is
how the pants length relates to the area. Generalizing his work with Sepp\"al\"a, 
Buser proved that any (not necessarily hyperbolic) genus $g$ surface with area $A$ has pants length at most 
$C g^{1/2} A^{1/2}$ \cite{BuserBook}. It's easy to give examples of surfaces with area $A$ and pants length at least
$C A^{1/2}$. Generalizing his conjecture for hyperbolic metrics, Buser conjectured that any metric
on a surface with area $A$ should have pants length at most $C A^{1/2}$. 
Recall that the systolic inequality for surfaces says that any surface of genus $g \ge 1$ and area $A$ has a non-trivial
curve of length at most $C A^{1/2}$. If it's true, Buser's conjecture about pants lengths would
greatly strengthen the systolic inequality. For an introduction to systolic geometry, see Chapter 4 of \cite{gromov}.

In both cases, hyperbolic and combinatorial, our results seem to be
based primarily on genus. Just as spheres with cusps have much
smaller pants length than high-genus hyperbolic surfaces with the same
area, small-genus random surfaces, e.g.\ random combinatorial surfaces
conditioned to have genus 0, might have simpler geometry than
high-genus random surfaces. 

The study of small-genus random surfaces is an active area in
probability theory especially in recent years, and it has been an
active area in physics for a long time. A lot of study has been paid
to random metrics on $S^2$, and one of the key ideas is that there is
essentially only one really natural random metric on $S^2$. Following
this idea, one conjectures that different procedures for producing a
random metric on $S^2$ should give the same result. In addition to
the example above, where one glues $N$ triangles together and rejects
any result which isn't a sphere, 
one can begin with the standard metric on $S^2$ and change it by a
randomly chosen conformal factor (the conformal factor being a
Gaussian free field). Conjecturally, for large $N$, the probability
distribution coming from gluing triangles converges to the probability
distribution for the random conformal factor. For an introduction to
these issues, see the survey article \cite{LeGall}.


There are many open questions related to the work in this paper. The
central problem, understanding the maximal pants length of a genus $g$
hyperbolic surface, remains wide open. We found random surfaces to be
useful in studying the total pants length, and there are also many
questions about pants length for random surfaces. For instance, what
is the average pants length of a random genus $g$ surface, and is the
average pants length close to the Bers constant? Do the pants lengths
of random genus $g$ surfaces concentrate near a single value? One
also has the analogous questions for total pants length.

In several ways, arithmetic surfaces have properties similar to random surfaces. How
does the pants length of an arithmetic surface compare with the pants length of
a random surface? In particular, is it true that an arithmetic surface has pants length
at least $g^{1/6 - \varepsilon}$?

Many of the same questions arise for random combinatorial surfaces
along with some new ones. Most random combinatorial surfaces have
large genus, but we can also consider the set of combinatorial
surfaces where the genus $g$ is fixed, but the number of triangles $N$
increases. If $g$ is small, say $g = 0$, and we look at large values
of $N$, then we get something quite similar to a ``random planar map''
studied in probability and physics - see \cite{LeGall}. If we take
$g$ large and $N$ to be roughly $4g$, it seems that we get some kind
of approximation to the moduli space of hyperbolic metrics. There are
many questions about the geometry of such a random surface -
what is its systole, first eigenvalue of Laplacian, or its Uryson
width? What are its isoperimetric properties? How many balls of
various radii are needed to cover it? What are its pants length and
total pants length? And so on.

\vskip5pt

In the first section of the paper, we review the topology of pants decompositions. In the second
section, we review some key background theorems about the Weil-Petersson volume form, and
we use them to prove Theorem 1. In the third section, we introduce the combinatorial moduli
space and prove Theorem 2. In the last section, we mention some open problems.\\

\noindent {\bf Notation:} Many of the numbers we will be concerned with are super-exponential. 
For two numbers $A(x)$ and $B(x)$ that depend on a variable $x$, $A(x)\approx B(x)$ (resp. $A(x) \gtrsim B(x) $) will
mean that they are equal (resp.\ the inequality holds) up to an
exponential factor in $x$. For example, by Stirling's inequality, $g! \approx g^g$.

\section{Topological types of pants decompositions}

Pants decompositions come in different topological types. 
Let us fix a surface $\Sigma$. A pants decomposition determines a
trivalent graph where each pair of pants corresponds to a vertex and two vertices
are joined by an edge if the corresponding pants share a boundary. 
(This trivalent graph may have multiple edges or loops.) We call this
graph the {\em topological type} of the pants decomposition.
We say that two pants decompositions are topologically equivalent
if their topological types are isomorphic graphs.
It's straightforward to check that if two pants decompositions are
topologically equivalent, then there is a diffeomorphism of $\Sigma$ taking
one to the other.

For example, if $\Sigma$ is a surface of genus 2, then it has two
topological types of pants decomposition. The two types each correspond
to trivalent graphs with two vertices. In the first case, there are three edges
that go between the two vertices. In the second case, each vertex has one
edge connecting it to the other vertex and one loop connecting it to
itself.

The first result that we need is an estimate for the number of different
topological types of pants decomposition on a surface of genus $g$.
In \cite{Bollobas}, Bollobas gave a precise asymptotic formula
for the number of trivalent graphs. We need only the following cruder version 
of Bollobas's formula.

\begin{lemma}\label{lem:countEquivs} (Bollobas)
 There are $\approx n^{n}$ trivalent graphs with $2n$ vertices.
\end{lemma}

This lemma is cruder than Bollobas's result, and it
also has a simpler proof. For reference, we include a proof here.

\begin{proof}
 We start with $2n$ labeled tripods and
 consider all the ways to glue them together to produce a trivalent
 graph. The tripods have total degree $6n$, and there are
 $$\frac{(6n)!}{(3n)!2^{3n}} \approx n^{3n}$$ 
 ways of dividing the $6n$ half-edges into pairs, each of which corresponds to a
 trivalent graph with vertices numbered $v_1, \dots, v_{2n}$. An
 unlabeled trivalent graph occurs many times in this collection. The
 permutation group $S_{2n}$ acts on the set of labeled graphs by
 permuting the labels, and each orbit of the permutation group
 consists of isomorphic graphs. The number of equivalence classes is
 thus at most the number of orbits. If $G$ is a labeled graph,
 recall that its orbit has $\Card S_{2n}/\Card S_{2n}^G$ elements, where
 $S_{2n}^G$ is the stabilizer of $G$. The stabilizer of $G$
 consists of permutations of the vertices which lead to an isomorphic
 graph. We can describe such a permutation in terms of the image of
 a basepoint in $G$ and a permutation of the neighbors of each
 vertex, so
 $$1\le \Card S_{2n}^G\le 2n 6^{2n} \lesssim 1.$$
 Hence each orbit has $\approx n^{2n}$ elements, and the number of
 orbits is $\approx n^{n}$.
\end{proof}

Consequently, a pants decomposition of a genus $g$ surface has one of
$\approx g^{g}$ possible topological types.

\section{The moduli space of hyperbolic metrics}

In this section we show that a random hyperbolic metric on a genus $g$ surface
has total pants length at least roughly $g^{7/6}$ with very high probability.
To begin, let us define what we mean by a random metric and make a precise
statement.

We denote the moduli space of closed hyperbolic surfaces of genus $g$ by $\MM_g$.
The Weil-Petersson metric is a Riemannian metric on $\MM_g$. We use the volume
form of the Weil-Petersson metric to define volumes on moduli space. By renormalizing
the Weil-Petersson volume form, we get a probability measure on moduli space.
We take random metrics with respect to this probability measure.

\begin{theorem} For any $\varepsilon > 0$, a random metric in $\MM_g$ has total pants
length at least $g^{\frac{7}{6} - \varepsilon}$ with high probability: the probability tends to
1 as $g \rightarrow \infty$.
\end{theorem}

Let us indicate the plan of our proof. 

First, we observe that without loss of generality, we may assume that the curves in
the pants decomposition are closed geodesics. Suppose we begin with a pants decomposition
of a hyperbolic surface $\Sigma$. By definition, the pants decomposition consists
of disjoint embedded curves $\gamma_1, ..., \gamma_{3g-3}$ so that each component
of the complement is a three-holed sphere. This is equivalent to just requiring that
the curves $\gamma_i$ are disjoint and non-parallel: no two curves among them
bound an annulus. Following a standard trick, we `tighten' the curves $\gamma_i$:
we replace each curve $\gamma_i$ with a closed geodesic $\bar \gamma_i$ in its free homotopy class. By
standard arguments in hyperbolic geometry, the closed geodesics will be embedded and disjoint,
and no two of them will bound an annulus. Hence the curves $\bar \gamma_i$ give a 
pants decomposition of $\Sigma$ also, and they have smaller lengths than the curves
$\gamma_i$.

From now on we assume that the curves in the pants decomposition are closed
geodesics. So each three-holed sphere in the pants decomposition has a hyperbolic
metric with geodesic boundary. We call such a pair of pants with such a metric
a {\it hyperbolic pair of pants}. The classification of hyperbolic pairs of pants is easy to descibe: 
the boundary curves may have any positive lengths, and for each
choice of lengths there is a unique metric.

Now our plan consists of describing all the ways to glue together hyperbolic pairs of pants
to make a closed hyperbolic surface, and then estimating the volume in moduli space
covered by these surfaces. 


To prove our theorem, we use two fundamental facts about the
Weil-Petersson volume form. Recall that the Teichm\"uller space $\T_g$
denotes the space of hyperbolic metrics on a fixed surface of genus
$g$, where two metrics are equivalent if they are related by an
isometry {\it isotopic to the identity}. The moduli space $\MM_g$ is
the quotient of Teichm\"uller space $\T_g$ by the action of the
mapping class group. The Weil-Petersson metric on Teichm\"uller space
is a non-complete K\"ahler metric with negative sectional curvature,
and is very much related to the hyperbolic geometry of surfaces.
Although it is a very natural metric to consider, it is quite
technical to define, so we refer the reader to \cite{Wolpert} for
details. The Weil-Petersson metric is invariant under the action of
the mapping class group, so it descends to a metric on moduli space.

\newtheorem{bthm}{Background Theorem}

We will need the following result of Schumacher and Trapani \cite{SchTra}.

\begin{bthm}[Asymptotic volume growth]\label{thm:volumeWP} The volume
 of moduli space $\MM_g$ grows (up to an exponential factor) like
 $g^{2g}$; i.e.,
 $$\Vol \MM_g\approx g^{2g}.$$
\end{bthm}

This result was an improvement of previous lower \cite{Penner} and upper bounds \cite{Grush}.

The second background theorem expresses the Weil-Petersson volume form
in a set of Fenchel-Nielsen coordinates. Before stating the result, we
quickly recall Fenchel-Nielsen coordinates.

Fix a pants decomposition of the genus-$g$ surface. We denote the
curves in the pants decomposition by $\gamma_1, ..., \gamma_{3g-3}$.
Recall that $l_i$ and $\tau_i$, the length and twist parameters,
define coordinates on the Teichm\"uller space $\T_g$. The length
parameter $l_i$ measures the length of the shortest curve homotopic to
$\gamma_i$ in the given metric; this is a positive real number. The
twist parameter $\tau_i$ measures the twist in the gluing across this
geodesic; this is a real number measured in units of length.
Different length and twist parameters may correspond to the same point
in $\MM_g$; for instance, replacing $\tau_i$ by $\tau_i\pm l_i$
yields a metric isometric to the original one by a Dehn twist around
$\gamma_i$.

The volume form for the Weil-Petersson metric has a very simple form in terms of these coordinates \cite{Wolpert}.

\begin{bthm}[Wolpert] In Fenchel-Nielsen coordinates, the volume form of the Weil-Petersson metric
is simply the standard volume form $dl_1 \wedge ... \wedge d l_{3g - 3} \wedge d \tau_1 \wedge ... \wedge
d \tau_{3g-3}$.
\end{bthm}

The region in moduli space with total pants length less than $L$ {\it in our fixed pants decomposition} is covered
by the following region of Teichm\"uller space:

$$S=\{ (l_i, \tau_i) \in T_g | \sum_i l_i \le L, 0 \le \tau_i \le l_i \} $$

The Weil-Petersson volume of this set is the same as its volume in the standard Euclidean metric on $\mathbb{R}^{6g-6}$ and it is not hard to estimate.

\begin{lemma} 
 If $1 \le L \le exp(g)$, then 
 $\Vol S \lesssim (L/g)^{6g}$, where $\lesssim$ is taken with respect
 to $g$.
\end{lemma}

\begin{proof} First consider the $(3g-3)$-dimensional simplex defined by the inequalities $0 < l_i$, $\sum_i l_i \le L$.
We denote this simplex by $\Delta_L$. By Fubini's theorem, 
$$\Vol S= \int_{\Delta_L} \prod_{i=1}^{3g-3} l_i . $$
By the arithmetic-geometric mean inequality, 
$$\prod l_i \le (\frac{L}{3g-3} )^{3g-3} \lesssim (\frac{L}{g})^{3g}.$$
Hence $\Vol S\lesssim \Vol (\Delta_L) (L/3g)^{3g}$.

The volume of the simplex $\Delta_L$ may be calculated inductively using the formula for the volume of
a pyramid. It is equal to $\frac{L^{3g-3}}{(3g-3)!} \lesssim (L/g)^{3g}$. \end{proof}

\noindent {\bf Remark.} The calculation in this lemma is basically sharp: the region of Teichm\"uller space above has
volume $\approx (L/g)^{6g}$. The region of moduli space covered by this region of Teichm\"uller space
has volume $\lesssim (L/g)^{6g}$. (The volume in moduli space may be much smaller if the covering map
is highly non-injective. We don't know how to estimate this effect.)

Now let $\MM_g(\le L) \subset \MM_g$ denote the subset of hyperbolic metrics that admit pants decompositions of total
length $\le L$.

If $E$ denotes a topological type of pants decomposition, we let $\MM_g(\le L, E) \subset \MM_g(\le L)$
denote the the subset of hyperbolic metrics that admit a pants decomposition of type $E$ and total length
$\le L$. For each $E$, the calculation in Fenchel-Nielsen coordinates shows that the volume of $\MM_g(\le L, E)$
is $\lesssim (L/g)^{6g}$. There are $\approx g^g$ different topological types $E$. Every pants decomposition belongs
to one of these $\approx g^g$ types, and so the volume of $\MM_g(\le L)$ is $\lesssim L^{6g} g^{-5g}$.

As we saw above, the volume of $\MM_g$ is $\approx g^{2g}$. If we set $L = g^{\frac{7}{6} - \varepsilon}$, then we see
that the volume of $\MM_g(\le L) \lesssim g^{(2 - 6 \varepsilon) g}$. Recalling the definition of $\lesssim$ and $\gtrsim$,
we see that the volume of $\MM_g$ is at least $c e^{-cg} g^{2g}$, while the volume of $\MM_g(\le g^{(7/6) - \varepsilon})$
is at most $C e^{Cg} g^{(2 - 6 \varepsilon) g}$. So for $g$ sufficiently large, the volume of $\MM_g$ is much
larger than the volume of $\MM_g(\le g^{(7/6) - \varepsilon})$. This
proves Theorem 1.

\section{The combinatorial viewpoint}

If $N$ is even, one can construct a oriented surface by gluing
together $N$ triangles (we allow edges of the same triangle to be
glued together and allow edges to form loops). We call the
corresponding CW-complex a combinatorial surface. We declare two
combinatorial surfaces to be equivalent if there is a homeomorphism
which sends edges to edges and faces to faces and we define $\com{N}$
to be the set of equivalence classes of such surfaces with $N$
triangles. Gamburd and Makover \cite{GamMak} showed that as $N\to
\infty$, a random element of $\com{N}$ has genus at least
$(1/4-\varepsilon)N$ with high probability, so $\com{N}$ is a rough
combinatorial equivalent of $\MM_{N/4}$.

We think of each triangle in a combinatorial surface as a Euclidean equilateral
triangle with side length 1. In this way, each combinatorial surface in $\com{N}$ has a metric on it.
In particular, we can define its pants length and total
pants length. We also have a good notion
of a random combinatorial surface given by the counting measure
on $\com{N}$. Using these definitions, we get the following combinatorial version
of our main result.

\begin{theorem}
For any $\varepsilon > 0$, a random combinatorial surface in $\com{N}$ has total pants length
at least $N^{7/6 - \varepsilon}$ with probability tending to 1 as $N \rightarrow \infty$.
\end{theorem}

The proof of Theorem 2 is morally analogous to the proof of Theorem 1. It is more elementary,
because it does not rely on the Weil-Petersson metric, but there
are also some additional subtleties. 

One of the key observations in the proof of Theorem 1 was
that a hyperbolic surface of total pants length $L$ can be cut into {\it hyperbolic
pairs of pants} whose boundaries have total length at most $2L$. The main
subtlety in the proof of Theorem 2 is to find the right combinatorial analogue for this step.

Suppose we start with an arbitrary pants decomposition of a
combinatorial surface $\Sigma$. The curves $\gamma_i$ are arbitrary
curves, and so the pairs of pants in the decomposition do not have
combinatorial structures. To get a combinatorial decomposition, we
need the $\gamma_i$ to be combinatorial curves. For each $i$, we can
approximate $\gamma_i$ by a combinatorial curve $\bar \gamma_i$ which
is homotopic to $\gamma_i$ and has comparable length. At this point,
some problems appear: the curves $\bar \gamma_i$ need not be embedded
and need not be disjoint. If the $\bar{\gamma_i}$ are chosen
carefully, however, we can still express $\Sigma$ as the union of
combinatorial pairs of pants, glued along the $\bar{\gamma_i}$. A
combinatorial pair of pants will consist of triangles and edges, but
some of the edges may not border any of the triangles. These isolated
edges can be thought of as `infinitely thin' pieces of surface. We
will show that we can choose the $\bar \gamma_i$ to be combinatorial
geodesics, so that each one has minimal length compared to all
combinatorial curves in its homotopy class.

At this point, there is a further problem. A hyperbolic pair of pants is determined by
the lengths of its boundary curves. But there are many different combinatorial pairs of pants
with the same boundary lengths. When we count how many ways we can glue together
combinatorial pairs of pants with given boundary lengths, we get a large unwanted factor
coming from the different choices for a pair of pants with fixed boundary curves. The underlying
cause of this problem is that combinatorial geodesics --- unlike
hyperbolic ones --- are not unique. The solution to this problem is to consider
only special combinatorial pants decompositions which we call `tight pants decompositions'. 
We define these below. Roughly speaking, they are pants decompositions of minimal complexity
in an appropriate sense.

With this well-chosen definition, the analogy runs smoothly. In
Section 3.1, we prove that $\com{N}$ has
cardinality $\approx N^{N/2}$, but that for any fixed $g$, the number
of $N$-triangle combinatorial surfaces of genus $g$ grows only exponentially. In Section 3.2, we introduce combinatorial pants decompositions
and tight combinatorial pants decompositions. We show that any pants decomposition can be
improved to make a tight combinatorial pants decomposition. In Section 3.3, we count the
number of tight combinatorial pants decompositions of total length $\le L$.

\subsection{Counting combinatorial surfaces}\label{ss:moduli}
The goal of this part is to count surfaces that lie in our
combinatorial moduli space. Our main goal is to prove that the cardinality
of $\com{N}$ is $\approx N^{N/2}$. Since most surfaces in $\com{N}$ have
genus close to $N/4$, this is analogous to the fact that
the Weil-Petersson volume of moduli space of surfaces of genus $g$ is
$\approx g^{2g}$.

Over the course of our argument, we will need to consider
the set of combinatorial surfaces with a particular genus. Let $\commodgk$ denote all combinatorial surfaces of
genus $g$ with $k$ boundary components made from $N$ triangles.
Again, we will allow two edges
of the same triangle to be glued together, and again we consider
surfaces up to homeomorphisms preserving edges and faces.

For fixed $g$ and $k$, we will see that the cardinality of $\commodgk$ grows only exponentially with $N$.

We begin by studying the cardinality of $\com{N}$. 

\begin{lemma}[Combinatorial volume growth]
 $$\Card(\com{N}) \approx N^{N/2}.$$
\end{lemma}
\begin{proof}
 If $\Sigma\in \com{N}$, we can construct a trivalent
 graph with $N$ vertices by letting the vertices of the graph be the
 faces of $\Sigma$ and connecting vertices whose corresponding faces
 share an edge. Given a trivalent graph, we can
 construct a surface with the corresponding pattern of gluings, so by
 Lemma~\ref{lem:countEquivs}, $\Card(\com{N})\gtrsim N^{N/2}$. 
 
 On the other hand, many surfaces might correspond to the same graph,
 since the graphs do not record which of the three edges of each
 triangle is glued to which other edge. With this information, a
 graph uniquely identifies a surface, but there are only $6^N$ ways to
 add this information to the graph, so $\Card(\com{N})\lesssim
 N^{N/2}$ as well.
 \end{proof}

In contrast, the number of ways to triangulate a surface of a fixed
genus with many triangles grows only exponentially:

\begin{lemma}\label{lem:countLowGenus}
 For any $g\ge 0$, $k \ge 0$, 
 $$\Card({\rm Com}_{n} {\rm Mod}_{g,k}) \precsim e^{n}.$$

In other words,

 $$\Card({\rm Com}_{n} {\rm Mod}_{g,k}) \le C(g,k) e^{n}.$$

\end{lemma}

\begin{proof}

 In the special case that $(g,k)=(0,1)$ (i.e., for triangulations of
 a disk), this follows from a result of Brown \cite{Brown}; we will
 reduce the general case to the case of a disk.

 Brown counted the number of rooted simplicial triangulations of the disk, that
 is, triangulations with a marked oriented boundary edge such that
 the endpoints of each edge are distinct and no face is glued to itself, and showed
 that the number of such with $j+3$ boundary
 vertices and $k$ interior vertices is
 $$\frac{2(2j+3)!(4k+2j+1)!}{(j+2)!j!k!(3k+2j+3)!}.$$
 Restating this in terms of the number $n=j+2k+1$ of triangles, we
 find that
 $$\frac{2(2j+3)!(2n-1)!}{(j+2)!j!\bigl(\frac{n-j-1}{2}\bigr)!(\frac{3n+j-3}{2}\bigr)!},$$
 where we require that $0\le j \le n-1$ and $j\equiv n-1 \mod 2$.
 Call this number $\theta(n,j)$. Our combinatorial surfaces differ from
 Brown's in that the two endpoints of an edge may be identified, but any
 element of ${\rm Com}_{n} {\rm Mod}_{0,1}$ can be barycentrically
 subdivided twice to get a simplicial triangulation. We thus find that
 $$\Card({\rm Com}_{n} {\rm Mod}_{0,1}) \le \mathop{\sum_{j=0}^{6n-1}}_{j\equiv 6n-1\;(\text{mod 2})}\theta(6n,j).$$

 Since $j\le n-1,$ there is a $c$ such that
 $$\theta(n,j)\le e^{cn}
 \frac{(2n-1)!}{\bigl(\frac{n-j-1}{2}\bigr)!(\frac{3n+j-3}{2}\bigr)!}=e^{cn} (2n-1)\binom{2n-2}{\frac{n-j-1}{2}}\le e^{cn}2^{2n-2}(2n-1),$$
 so there is a $c'$ such that for all $n>0$, 
 $$\mathop{\sum_{j=0}^{n-1}}_{j\equiv n-1\;(\text{mod 2})}\theta(6n,j)\le
 e^{c'n}.$$

 Now consider $\commodgk$. If $\Sigma$ is a triangulated genus-$g$
 surface with $k$ holes and $(g,k)\ne (0,0)$, we can cut it along
 non-separating simple closed curves or simple arcs between boundary
 components until we get a disk; this takes at most $2g+k$ cuts. We
 can thus obtain any element of $\commodgk$ by performing at most
 $2g+k$ gluings on an element of ${\rm Com}_{n} {\rm Mod}_{0,1}$.
 Each gluing identifies two edge paths on the boundary, so a gluing
 is determined by the endpoints of the paths that are glued. There
 are at most $(3n)^4$ ways to perform each gluing, so as long as
 $(g,k)\ne (0,0)$,
 $$\Card({\rm Com}_{n} {\rm Mod}_{g,k}) \le e^{c'n}(3n)^{2g+k}
 \precsim e^{n}.$$ For the case that $(g,k)=(0,0)$, note that if
 $\Sigma$ is a triangulation of a sphere, then we can obtain a
 triangulation of the disc by cutting along an edge of $\Sigma$, so
 $$\Card({\rm Com}_{n} {\rm Mod}_{0,0}) \le e^{c'n} \precsim
 e^{n}$$
 as desired.
\end{proof}

\subsection{Combinatorial pants decompositions}

In this part, we define pants decompositions of
combinatorial surfaces and their lengths. We then focus our interest
on pants decompositions of minimal length, and we show that in the
isotopy class of such a pants decomposition there is always a 
pants decomposition of a particular type, called a tight pants decomposition.

A pants decomposition of a surface of genus $g$ is a
maximal set of disjoint and freely homotopically distinct non-trivial
simple closed curves. A pants decomposition always contains $3g-3$
curves, and its complementary region consists of a set of $2g-2$ three
holed spheres, or pairs of pants. To define a pants decomposition in
the combinatorial setting, we focus on these pairs of pants.

A combinatorial pair of pants will consist of a simplicial complex
equipped with some boundary curves. Let $\Delta$ be a simplicial
complex which is a deformation retract of a three-holed sphere $M_{0,3}$. If we
consider $\Delta$ as a subset of $M_{0,3}$, this implies that a
regular neighborhood of $\Delta$ is a three-holed sphere. The
boundary curves of this three-holed sphere correspond to simplicial
curves in the boundary of $\Delta$, and when $\Delta$ is equipped with
these boundary curves, we call it a {\em combinatorial pair of pants}.
These curves inherit an orientation from $M_{0,3}$. Note that
$\Delta$ need not be a manifold; for instance, it could be two
vertices connected by three edges. In general, $\Delta$ may contain
edges that are not boundaries of triangles. Such edges we call
stranded.

We can glue pairs of pants to get surfaces. If $P_1,\dots,
P_{2g}$ are combinatorial pairs of pants, each one has three
boundary components, and we can identify pairs of boundary components
which have the same length. Like its geometric analogue, there are
many ways to identify the same pair of boundary components, and these
ways differ by a shift; by fixing a basepoint on each boundary curve,
we can define twist parameters for the gluing. If we glue all of the
boundary curves in pairs, we obtain a complex $\bigcup P_i/\sim$;
this may not be a surface, because there may still be edges which are
not part of a triangle. If $\bigcup P_i/\sim$ is 
isomorphic to a triangulated surface $\Sigma$, we call the collection
of the $P_i$, the gluing instructions, and the isomorphism a {\em
 combinatorial pants decomposition} of $\Sigma$. The boundary curves
of the $P_i$ project to simplicial curves on $\Sigma$; we call these the
boundary curves of the pants decomposition. If the boundary curves
have minimal (combinatorial) length in their free homotopy classes, we
say that the pants decomposition is minimal.

We will show that a geometric pants decomposition gives rise to a
minimal combinatorial pants decomposition. Recall that a Lipschitz
closed curve on a triangulated surface is homotopic to a simplicial
curve whose length is bounded by a constant times the length of the
original curve:

\begin{lemma} Let $\alpha:S^1 \to \Sigma$ be a Lipschitz curve on a
 triangulated surface $\Sigma$. There is a simplicial curve
 $\lambda$ on $\Sigma$ which is homotopic to $\alpha$ and whose
 length satisfies
 $$\ell(\lambda) \leq 2 \ell(\alpha).$$
\end{lemma}

\begin{proof}
 For any $ \varepsilon>0$, we can perturb $\alpha$ to get a smooth curve
 $\alpha'$ of length $\ell(\alpha')\le \ell(\alpha)+ \varepsilon$ which
 avoids vertices of $\Sigma$ and intersects its edges transversely.
 The edges of the triangulation cut $\alpha'$ into finitely many
 arcs, each contained in a single triangle. Each arc $a$ cuts the
 boundary of its triangle into two arcs. We homotope the arc $a$ to
 the shortest of these two boundary arcs (say $b$). By elementary
 euclidean geometry, we have $\ell(b)\leq 2 \ell(a)$. As such, the
 resulting curve $\beta$ satisfies $\ell(\beta) \leq 2 \ell(\alpha)$.

 The resulting curve is not necessarily a simplicial curve as it
 may go partway along an edge and then backtrack. A further homotopy
 removes this backtracking and decreases the length. The resulting
 curve $\lambda$ is now a simplicial curve and has length at most
 the length of $\beta$. This proves the lemma.
\end{proof}

We can now focus our attention on the equivalent statement for full pants decompositions.

\begin{proposition}
 Let $\Sigma$ be a triangulated surface and let $\alpha_1,\dots,
 \alpha_{k}:S^1\to \Sigma$ be the boundary curves of a pants decomposition for
 $\Sigma$. There is a minimal combinatorial pants decomposition of
 $\Sigma$ with boundary curves $\lambda_1,\dots, \lambda_{k}:S^1\to \Sigma$ such
 that for all $i$, $\lambda_i$ is homotopic to $\alpha_i$ and
 $\ell(\lambda_i)\le 2 \ell(\alpha_i)$.
\end{proposition}
\begin{proof}
 The first step is to approximate the $\alpha_i$ by simplicial
 curves. For all $i$, let $\beta_i$ be a simplicial closed curve of
 minimal length which is homotopic to $\alpha_i$. In general, the
 $\beta_i$ may share edges; we will subdivide $\Sigma$ and make them
 disjoint.

 We first duplicate the edges of $\Sigma$ so that the $\beta_i$ do
 not share edges. If $e=(x,y)$ is an edge of $\Sigma$ which occurs
 $n\ge 2$ times in the $\beta_i$, we replace $e$ with $n-1$ bigons
 glued edge-to-edge. We make no changes
 to edges which do not occur or occur only once. Now we replace the
 vertices of $\Sigma$; if a vertex has degree $d$, we replace it with
 a $d$-gon which we call a {\em vertex polygon}. Each incoming edge
 connects to one vertex of this $d$-gon; this makes the bigons
 created in the previous step into rectangles which we call {\em edge
 rectangles}. Each edge rectangle has two edges which are shared
 by vertex polygons and two edges which connect vertex polygons; we
 call the edges which connect vertex polygons {\em long edges}. Call
 the resulting complex $\Sigma'$. This complex is homeomorphic to
 $\Sigma$, and there is a natural map $p:\Sigma'\to \Sigma$ which
 collapses the edge rectangles to edges and the vertex polygons to
 vertices; this map sends long edges to edges of $\Sigma$
 homeomorphically.
 
\begin{figure}[htbp]
 \centering
 \includegraphics[width=8 cm]{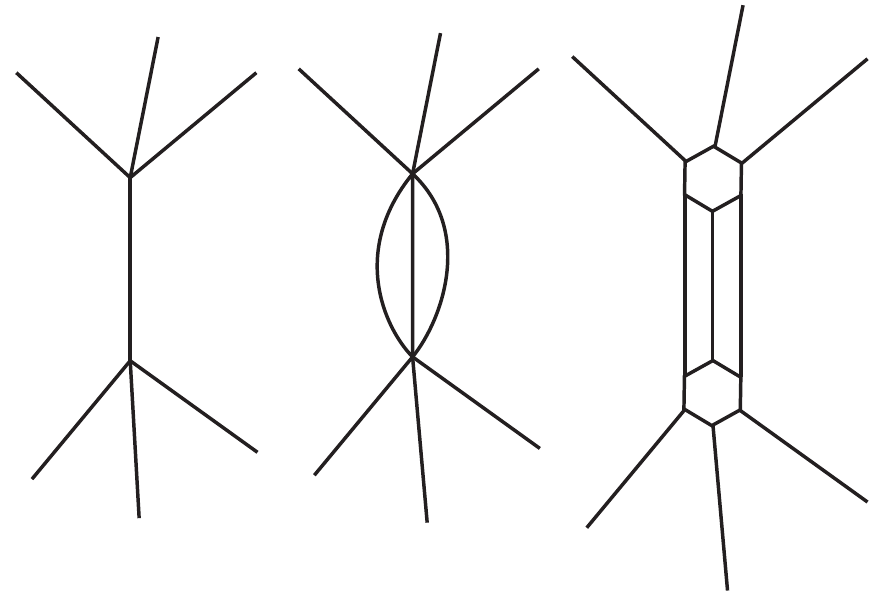}
 \caption{From $\Sigma$ to $\Sigma'$}
 \label{vertex}
\end{figure}

 We call a curve {\em semi-simplicial} if it consists of alternating
 long edges and curves in vertex polygons; the image of a
 semi-simplicial curve under $p$ is thus a simplicial curve. If
 $\gamma$ is a semi-simplicial curve, we define its length by
 $\ell(\gamma):=\ell(p(\gamma))$. We can lift each of the $\beta_i$
 to semi-simplicial curves in $\Sigma'$ by replacing edges of
 $\beta_i$ with corresponding long edges and replacing vertices of
 $\beta_i$ with curves in the vertex polygons. There are enough long
 edges in $\Sigma'$ that we can ensure that no long edge is used
 twice, so this gives us curves $\beta'_i$ in $\Sigma'$ which
 intersect only in the vertex polygons. We may assume that at most
 two curves intersect at a point and that all intersections are
 transverse.

 A standard argument allows us to perform surgeries to make these
 curves disjoint. If two curves $\beta'_i$ and $\beta'_j$ intersect
 (where possibly $i=j$), the fact that the two curves are homotopic
 to disjoint curves (or, in the case that $i=j$, to a simple curve)
 implies that there is a pair of intersection points $x, y$, a
 segment $\gamma_1$ of $\beta'_i$, and a segment $\gamma_2$ of
 $\beta'_j$, such that $\gamma_1$ and $\gamma_2$ connect $x$ and $y$
 and are homotopic relative to their endpoints.
 Because $\beta'_i$ and $\beta'_j$ have minimal length, we have
 $\ell(\gamma_1)=\ell(\gamma_2)$. We can modify $\beta'_i$ and
 $\beta'_j$ by swapping $\gamma_1$ and $\gamma_2$ and deforming the
 resulting curves near $x$ and $y$ to eliminate those two
 intersection points. This operation reduces the number of
 intersection points by two, so we can repeat it to eliminate all
 intersection points. Since we have only swapped subsegments of
 curves and performed homotopies inside vertex polygons, the
 resulting curves are still semi-simplicial and still have minimal
 length; call them $\beta''_1,\dots, \beta''_k$.

 We will get a combinatorial pants decomposition of $\Sigma$ by
 cutting $\Sigma'$ along these curves and collapsing edge rectangles
 and vertex polygons. The curves $\beta''_1,\dots, \beta''_k$ are
 the boundary curves of a geometric pants decomposition of $\Sigma'$
 into subsurfaces $P_1,\dots,P_{2k/3}$. Each of these subsurfaces is
 homeomorphic to a pair of pants (i.e., a three-holed sphere), and is
 the union of edge rectangles, subsets of the vertex polygons, and
 faces of $\Sigma'$ which come from triangles of $\Sigma$. If we
 collapse each edge rectangle in $P_i$ to an edge and each connected
 component of a vertex polygon to a vertex, we obtain a complex
 $P'_i$ which comes with a map $p_i:P'_i \to \Sigma$. The boundary
 curves of $P_i$ correspond to simplicial curves in $P'_i$, and with
 these boundary curves, $P'_i$ forms a combinatorial pair of pants.
 Furthermore, gluing the $P'_i$ along these boundary curves
 reconstructs the original surface $\Sigma$, making them a
 combinatorial pants decomposition of $\Sigma$. The boundary curves
 of this pants decomposition are $\lambda_i:=p(\beta''_i)$; as
 required, $\lambda_i$ is homotopic to $\alpha_i$, and since it has
 minimal combinatorial length, its length is no more than a constant
 factor larger than that of $\alpha_i$.
\end{proof}

Let $\Sigma$ be a genus $g$ surface and let $P_1,\dots,
P_{2g-2}$ be the pants in a combinatorial pants decomposition of
$\Sigma$.
We can view a combinatorial pair of pants $P$ as a collection of {\em
 clusters} connected by {\em strands}; indeed, we will construct a
graph $G$ whose vertices correspond to clusters of $P$ and whose edges
correspond to strands. We construct this graph as follows (see Fig.\
\ref{fig:strandCluster}):
\begin{enumerate}
\item For all vertices $v\in P$, if the link of $v$ has $d$ connected
 components and more than one is an interval, replace $v$ by a
 star with $d$ edges.
\item Shrink paths of edges to single edges.
\item Shrink groups of triangles which are glued along edges to single
 vertices.
\end{enumerate}
Each vertex of
$G$ corresponds to a group of triangles (indeed, a submanifold of $P$
with boundary)
or a single point; we call these clusters. We will call a
single-point cluster {\em degenerate}. Each edge corresponds to a
path of stranded edges of $P$ (possibly a path of length zero); we
call these strands. Since each cluster corresponds to a vertex of
$G$, we can define the {\em degree} of a cluster to be the degree of
the corresponding vertex.

\begin{figure}
 \centering
 \def\svgwidth{\columnwidth}
 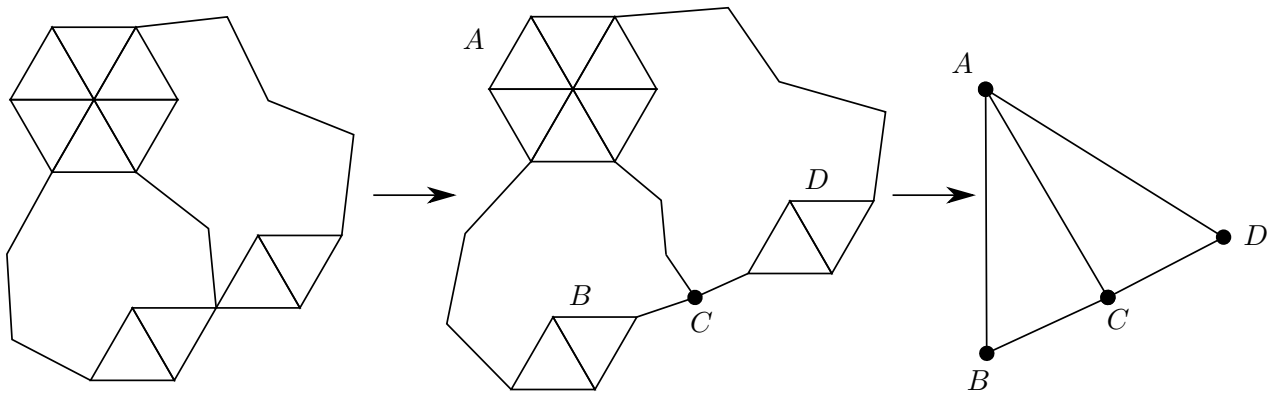
 \caption{The cluster graph. This pair of pants has four disk-type
 clusters, one of which ($C$) is degenerate and two of which ($B$,
 $D$) are loose disks.}
 \label{fig:strandCluster}
\end{figure}

The interior of a cluster can be homeomorphic to a disk, a
cylinder or a three holed sphere, and we call it disk-type,
cylinder-type, or three-holed-sphere-type accordingly. If a cluster
is a single point, we say it is disk-type. If the pants decomposition was
minimal, then a disk-type cluster can have degree two, three, or four.
A cylinder-type cluster has degree one or two, and a
three-holed-sphere-type cluster has degree zero. A pair of pants $P$ is called tight if none of
its disk-type clusters have degree 2. Such a cluster
will be called a loose disk. A minimal length
pants decomposition is called tight if all of its pants are tight,
i.e., do not contain any loose disks. 

\begin{figure}[htbp]
 \centering
 \includegraphics[width=8 cm]{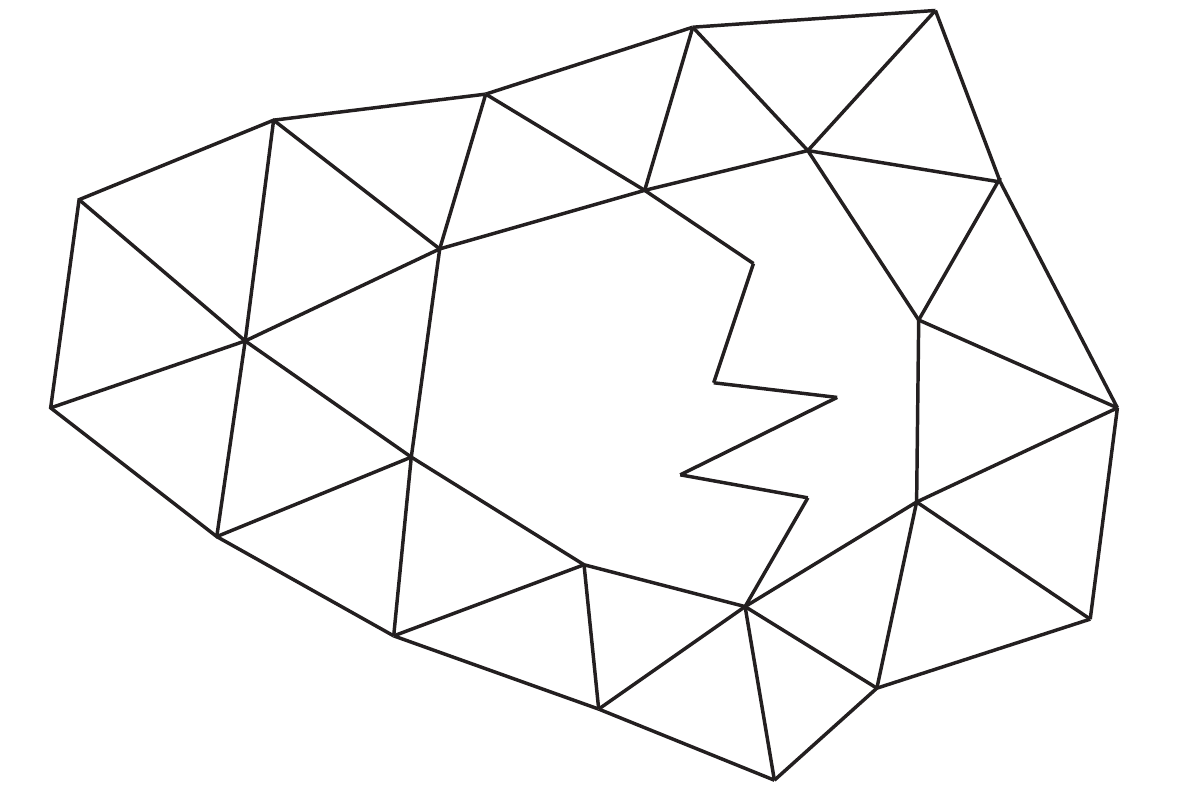}
 \caption{A pair of pants with a annulus-type cluster}
 \label{cylinder}
\end{figure}


We are now able to introduce the
main result of this section.
\begin{lemma} Any combinatorial surface $\Sigma$ admits a minimal length tight pants decomposition.
\end{lemma}
\begin{proof}
We need to show that if a minimal pants decomposition has pants with
loose disks we can isotope the curves to remove the disk from the pair
of pants without increasing length. We further need to make sure that
by doing so we are not just moving a loose disk somewhere else, and
that in fact we will have reduced the number of loose disks.

\begin{figure}[htbp]
 \centering
 \includegraphics[width=8 cm]{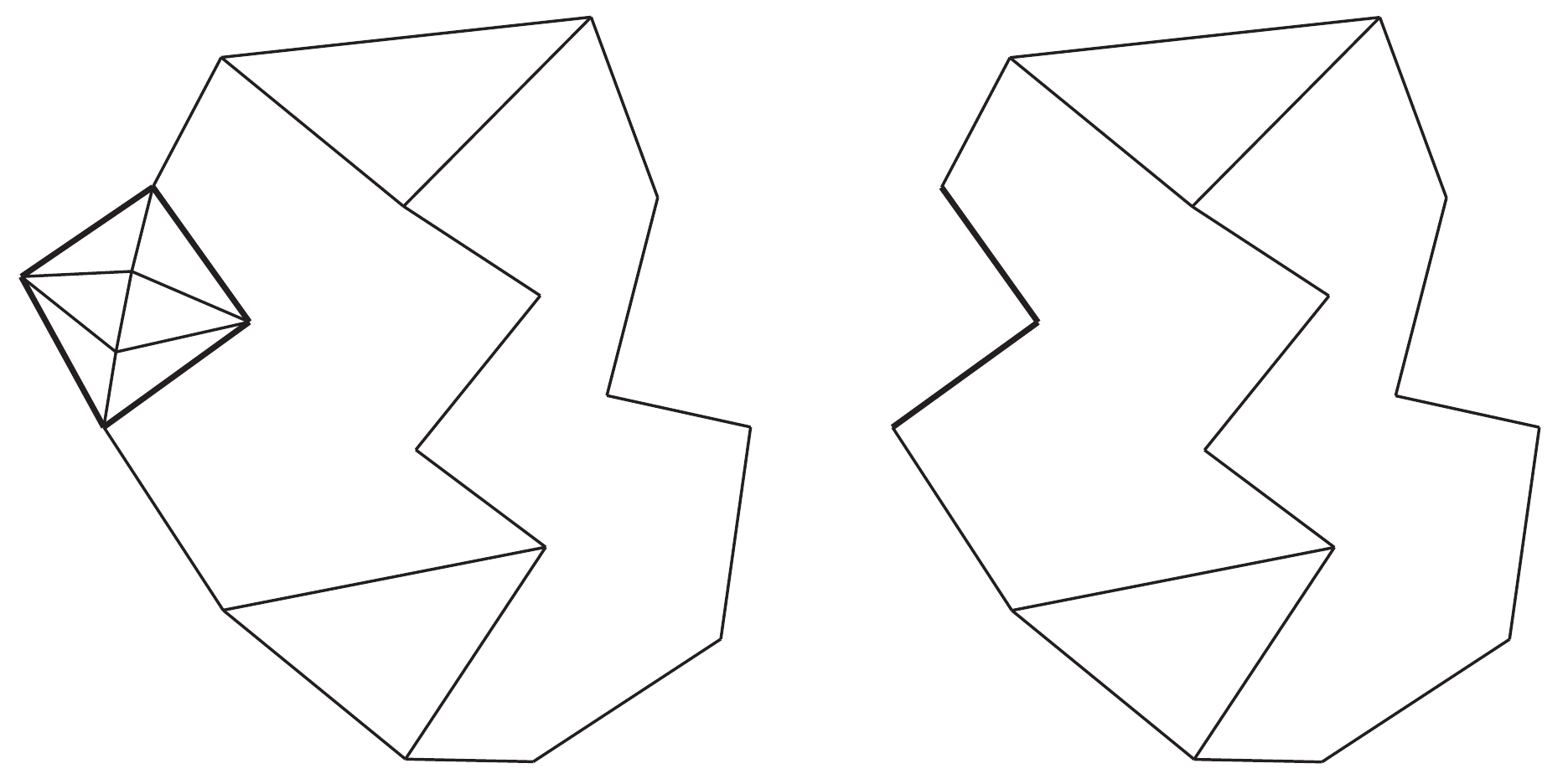}
 \caption{Removing a loose disk to get a tight pair of pants}
 \label{tight}
\end{figure}

The key technique in this proof is sliding a loose disk from one pair
of pants to another. Let $D$ be a loose disk which is part of a pair
of pants $P$ in a minimal pants decomposition, and say that $P$ has
boundary curves $\gamma_{a}:S_{a}\to P$, $a=1,2,3$. Since $D$ is a
loose disk, there are two strands which enter $D$, say at vertices
$x_1$ and $x_2$. These two vertices divide the boundary of $D$ into
two paths, and since the boundary curves of $P$ have minimal length,
the two paths have equal length. We denote them $c,c':[0,\ell]\to P$,
and we can choose $n$ and $n'$ so that $c$ is a subsegment of
$\gamma_{n}$ and $c'$ is a subsegment of $\gamma_{n'}$. If $P$ is
glued to $Q$ along $\gamma_{n}$, then we can slide that common
boundary curve over $D$ to transfer $D$ from $P$ to $Q$. This has the
effect of cutting $D$ out of $P$ (leaving $c'$) and gluing it on to
$Q$ along $c$, and it produces a new pants decomposition of $\Sigma$.
We call this process {\em sliding $c$ to $c'$}. Furthermore, since
$c$ and $c'$ must have equal lengths, the lengths of the boundary
curves are unchanged and the new pants decomposition is also minimal.

After such a slide, the pants decomposition differs from the original only
around $D$. All the clusters of the original decomposition have
counterparts in the new one, except possibly for clusters in $P$ and
$Q$; the move deleted one cluster from $P$ and added at most one to
$Q$, depending on whether $D$ was glued to a cluster or to a strand.
If $D$ was glued to a cluster, then the move reduced the number of
loose disks by one. Otherwise, the slide glues $D$ to a strand of
$Q$. In this case, after the slide $D$ becomes a loose disc in $Q$ bounded by
segments $c_1$ and $c'_1$.

We will inductively define a sequence $X_0,X_1,\dots,X_k$ of pants
decompositions which all differ by slides. Each $X_i$ except $X_k$
will have a loose disk $D_i$ in a pair of pants $P_i$. This disk will
be
isomorphic to $D$ and which is bounded by two curves $c_i:[0,\ell]\to
P_i$ and $c'_i:[0,\ell]\to P_i$. Recall that $\Sigma$ is a quotient
of the pairs of pants in $X_i$; let $\mu_i:P_i\to \Sigma$ be the
restriction of the quotient map to $P_i$. We will require that
$\mu_i\circ c_i:[0,\ell]\to \Sigma$ be the same curve for all $i<k$
and likewise for $\mu_i\circ c'_i:[0,\ell]\to \Sigma$, and we will
define $f:=\mu_i\circ c_i$ and $f':=\mu_i\circ c'_i$. 

Let $X_0$ be the original decomposition of $\Sigma$; this has a loose
disk $D_0\cong D$ bounded by $c_0:=c$ and $c'_0:=c'$. We construct
$X_{i+1}$ from $X_i$ by sliding $c_i$ to $c'_{i}$. If this reduces
the number of loose disks, we stop, letting $k=i+1$; otherwise, $D_i$
corresponds to a loose disk of $X_{i+1}$, bounded by $c_{i+1}$ and
$c'_{i+1}$, and we continue. We claim that this process eventually
stops.

By way of contradiction, say that the process does not stop.
Sliding $c_i$ to $c'_i$ affects the boundary curves of $X_i$ by
replacing an occurrence of $f$ by $f'$. If the process does not
stop, then we can replace $f$ by $f'$ infinitely many times. In
particular, each edge of $\Sigma$ occurs as many times in $f$ as it
does in $f'$, so each edge of $\Sigma$ occurs an even number of times
in $\mu(\partial D)$ (indeed, either 0 or 2). Consequently,
$\mu(D)=\Sigma$. Since $\mu$ is injective on the interior of
$D$, this means that we can obtain $\Sigma$ by gluing the edges of
$D$ together.

If $w=c(0)$ or $w=c(\ell)$, we call $w$ an endpoint of $D_0$. We
claim that for all $v\in \Sigma$, $\mu^{-1}(v)$ contains at most 2
non-endpoint vertices of $D$. Say $v$ is such that
$\{w_1,w_2,w_3\}\subset \mu^{-1}(v)$ for some 3 distinct non-endpoint
vertices $w_1,w_2,w_3\in D$. The link of $v$ is a circle, and it
contains 3 intervals corresponding to the links of the $w_i$. Let
$S\subset \link v$ be the complement of the interiors of these
intervals; this consists of three connected components, $S_1, S_2$,
and $S_3$. If a length-2 segment of a boundary curve of $X$ passes
through $v$, there are $j$ and $k$ such that it approaches $v$ from
the direction of $S_j$ and leaves $v$ in the direction of $S_k$. We
call this a $jk$-segment. Note that the only $jk$-segments of
$\partial D$ with $j\ne k$ are those centered at the $w_i$. Since
$D$ and the $S_a$'s are unaffected by repeatedly sliding $c$ to $c'$, we can
discuss $jk$-segments of the boundary curves of the $X_i$ too.

One of the $w_i$ is in the image of $c$; say $w_1=c(x)$, so that $f$
passes through $v$ at $x$, and number the $S_i$ so that this is a
$12$-segment. Replacing $f$ by $f'$ deletes this segment, and we
claim that replacing $f$ by $f'$ decreases the number of $12$-segments
by one. The path $f'$ has no $12$-segments, so it only remains to
check that replacing $f$ by $f'$ can't introduce new $12$-segments
centered at the endpoints of $D$. But if $f(0)=f'(0)=v$ or
$f(\ell)=f'(\ell)=v$, then $f$ and $f'$ both leave $v$ in the
direction of the same $S_i$, so a $jk$-segment centered at an endpoint
remains a $jk$-segment when $f$ is replaced by $f'$. Since the number
of $12$-segments in $X$ is finite, the process terminates after a
finite number of slides.

Thus, $\Sigma$ can be obtained by gluing a disc to itself along its
edges; the resulting gluing has one face, namely $D$, $\ell$ edges,
and at least $(\ell-1)$ vertices. Thus, if the process does not
terminate, then $\Sigma$ has genus at most 1, which is a contradiction.
\end{proof}

\subsection{Counting tight pants decompositions}

The goal of this part is to count the number of tight pants decompositions of bounded length. 

\newtheorem*{mainest}{Main Estimate}

\begin{mainest} There is a $c>0$ such that the number of triangulated
 surfaces in $\com{N}$ with genus $g$ and (tight) pants decompositions
 of total length at most $L$ is $\le exp(c N) g^g (L/g)^{6g}$.
\end{mainest}

First we will count the number of different tight pairs of pants with boundary curves of controlled length. Next we will count the number of ways of gluing these pants together into a surface.

\begin{lemma} \label{lem:countPants} There is a $c_0>0$ such that the number of tight pairs of
 pants with boundary curves of lengths $l_1, l_2,$ and $l_3$ and with
 $A$ triangles is $\le c_0 e^{c_0 A}$.
\end{lemma}

\begin{proof} A tight pair of pants consists of some clusters (which
 may be collections of triangles or may be single points) joined by some
strands. There are several combinatorial types,
including:
\begin{enumerate}
\item One annulus-type cluster, joined to itself with one strand.
\item Two disk-type clusters, joined by three strands running between
 them.
\item One disk-type cluster, joined to itself by two strands.
\end{enumerate}

What matters to us is that there are only a finite number of
combinatorial types. To see this, we begin by observing that each
cluster of triangles has to be a subsurface with genus 0 and at most 3
boundary components, so there are finitely many types of cluster. If
$c$ is the number of clusters, $s$ is the number of strands, and $b$
is the total of the first Betti numbers of the clusters in a pair of pants
$P$, then the first Betti number of $P$, $\beta_1(P)$ is given by
$\beta_1(P)=b+s-c+1$. If $P$ is tight, each disk-type cluster has
degree at least 3, so each cluster contributes at least $1/2$ to
$\beta_1(P)$. As such, there can be at most 2 clusters in a tight pair
of pants and at most 3 strands, so there are at most, say, $100$ types
of pairs of pants.

We can now prove that for each combinatorial type, there are $\precsim
e^{A}$ tight pairs of pants with $A$ triangles. Since there are a
finite number of combinatorial types, this will imply the lemma.

By Lemma~\ref{lem:countLowGenus}, there is a $c>0$ such that there are
at most $e^{cA}$ discs or annuli or three-holed spheres with $A$
triangles. Now 
to build a tight pair of pants, we have to add strands to the triangulated surfaces. We have to choose the
attaching points. There are at most six attaching points, and each attaching point has at most $3 A$ choices,
and so there are $\le (3A)^6$ choices of attaching points. Once we have chosen where to attach each strand,
the lengths of the strands are determined by the lengths $l_1, l_2,$
and $l_3$ of the three boundary circles. Thus the total number of
tight pants with fixed boundary lengths, a fixed combinatorial type, and $A$ triangles is bounded by
$100 e^{2 c A} (3A)^6 \precsim e^{A}$.
\end{proof}

\begin{remark}
 In hyperbolic geometry, there is a unique hyperbolic pair of
 pants for every choice of boundary lengths $l_1, l_2,$ and $l_3$.
 The closest combinatorial analogue of this phenomenon is the fact
 that tight pairs of pants with no triangles are determined by their
 boundary lengths. For every triple of lengths, $l_1, l_2, l_3 \in
 \mathbb{Z}$, there is a unique triangle-less pair of pants with the
 given lengths. If the lengths obey the triangle inequality, the
 graph looks like $\theta$ and if not the graph looks like a pair of
 glasses, that is, two circles connected by a line. If we are not careful when we add triangles, the number
 of pairs of pants explodes: if we add one triangle, we have $\approx l_1
 + l_2 + l_3$ different edges where we can put it, so we get many
 different pairs of pants. For this reason, we introduced tight
 pairs of pants; tightness restricts the possible places that a
 triangle can go. The number of tight pairs of pants with fixed
 boundary lengths is bounded by $\exp (A)$, with constant independent
 of the chosen boundary lengths. The $\exp(A)$ factor
 is fairly harmless, so tight pairs of pants are a good analogue of
 hyperbolic pairs of pants.
\end{remark}

Recall that a combinatorial pants decomposition consisted of a
collection of pants and some gluing information. Next we consider
combinatorial analogues of length and twist parameters (i.e.,
Fenchel-Nielsen coordinates) and bound the number of possible ways to
glue pants.

\begin{lemma} There is a $C$ such that the number of genus $g$
 combinatorial pants decompositions with total pants length $\le L$
 and total area $\le N$ is bounded by $\le C g^g (L/g)^{6g} exp(C N)$.
\end{lemma}

\begin{proof} In this proof, we will write $f(g,L,N) \lesssim
 h(g,L,N)$ to mean that there is some $c$ such that $f(g,L,N) \le
 ce^{cN}h(g,L,N)$ for all applicable $g,L,N$. Note that we may
 assume that $g\le N/4$, and so $e^g \lesssim 1$ as usual.

 As with hyperbolic pants decompositions, a genus-$g$ combinatorial
 pants decomposition has a topological type. By
 Lemma~\ref{lem:countEquivs}, the number of topological types of
 genus $g$ is $\approx g^g$.

Now we count tight pants decompositions with a fixed topological type. 
We first have to choose the lengths of the $3g -3$ boundary curves in the pants decomposition. How many 
ways can we choose positive integers $l_1, ..., l_{3g-3}$ so that $\sum l_i \le L$? This number is less than the volume
of the set $x_i \ge 0, \sum x_i \le L$. The volume of that simplex
can be computed by induction on the dimension; it has volume
$ \frac{1}{ (3g-3)! } L^{3g-3} \lesssim (L/g)^{3g}$.

For each choice of lengths, we next have to choose how many triangles to put in each pair of pants. Here we 
have to choose $A_1, ..., A_{2g-2} \ge 0$ with $\sum A_i = N$. The number of ways to choose $A_i$ is
exactly ${ N + 2g - 3 \choose 2g-3} \le 2^{N + 2g - 3} \le 4^N$. (Since $g \le N/4$.) So the number of ways
of choosing $A_i$ is $\lesssim 1$.

Next we have to choose a tight pants structure for each pair of pants
with the given area $A_i$ and the given boundary lengths. If $c_0$ is
the constant from Lemma~\ref{lem:countLowGenus}, then the number of ways
to do this is at most $\prod_i c_0 exp(c_0 A_i) \lesssim 1$.

Now we count the number of possible gluings. Since we already chose the topological
type, a gluing
is determined by its twist
parameters. For each of the $3g-3$ curves, the twist parameter is an integer in
the range $0 \le t_i \le l_i - 1$. The number of choices for the
twist parameters is 
$$\prod_{j= 1}^{3g-3} l_j \le ( \frac{L}{3g-3})^{3g-3} \lesssim (L/g)^{3g}.$$

Multiplying all of these together, we find that the number of possible
pants decompositions is
$$\lesssim g^g
\biggl(\frac{L}{g}\biggr)^{3g} \biggl(\frac{L}{g}\biggr)^{3g}\lesssim
g^g \biggl(\frac{L}{g}\biggr)^{6g},$$
as desired.
\end{proof}

In particular, the number of underlying
surfaces (up to simplicial isomorphism) is $\lesssim (L/g)^{6g}
g^g$. This proves the main estimate.

The total number of combinatorial surfaces in $\com{N}$ is
$\approx N^{N/2}$. If $N$ is sufficiently large and $L = N^{7/6 - \varepsilon}$,
then the number of surfaces in $\com{N}$ with genus $\ge 2$ and total pants length
at most $L$ is
$$\lesssim \sum_{i=2}^{(N+2)/4} e^{c N} (L/i)^{6i} i^i\lesssim e^{c N}
N^{(\frac{1}{6}-\varepsilon)\cdot \frac{6N}{4}}
(N/4)^{\frac{N}{4}}\approx N^{N/2-3\varepsilon/2}.$$ The number of
surfaces in $\com{N}$ with genus $<2$ is $\lesssim 1$ by
Lemma~\ref{lem:countLowGenus}, so for large $N$, most surfaces in
$\com{N}$ have no pants decomposition of total length $\le L$. This
finishes the proof of Theorem 2.
 
\bibliographystyle{amsalpha}
\providecommand{\bysame}{\leavevmode\hbox to3em{\hrulefill}\thinspace}
\providecommand{\MR}{\relax\ifhmode\unskip\space\fi MR }
\providecommand{\MRhref}[2]{%
  \href{http://www.ams.org/mathscinet-getitem?mr=#1}{#2}
}
\providecommand{\href}[2]{#2}

 \end{document}